\documentclass[a4paper]{amsart}
\usepackage[latin1]{inputenc}
\usepackage{amssymb}
\usepackage{amsmath}
\usepackage{mathrsfs}
\usepackage{eufrak}
\usepackage{amsthm}
\usepackage{amsfonts}
\usepackage{textcomp}
\usepackage{graphicx}
\usepackage[pdftex]{color}
\usepackage{paralist}
\usepackage[shortlabels]{enumitem}
\usepackage{hyperref}
\usepackage{comment}
\usepackage[arrow, matrix, curve]{xy}

\newtheorem*{corollary*}{Corollary}
\newtheorem{theorem}{Theorem}[section]

\newtheorem{corollary}[theorem]{Corollary}
\newtheorem{lemma}[theorem]{Lemma}
\newtheorem{proposition}[theorem]{Proposition}

\newtheorem*{claim*}{Claim}

\theoremstyle{definition}

\newtheorem*{maintheorem}{Maintheorem}
\newtheorem*{theorem }{Theorem}
\newtheorem{remark}[theorem]{Remark}
\newtheorem{example}[theorem]{Example}

\theoremstyle{remark}

\numberwithin{equation}{theorem}

\makeatletter
\renewcommand*\env@matrix[1][\
arraystretch]{%
  \edef\arraystretch{#1}%
  \hskip -\arraycolsep
  \let\@ifnextchar\new@ifnextchar
  \array{*\c@MaxMatrixCols c}}
\makeatother

\begin{document}

\title{Auslander-Gorenstein algebras, standardly stratified algebras and dominant dimensions}
\date{\today}



\author{Ren\'{e} Marczinzik}
\address{Institute of algebra and number theory, University of Stuttgart, Pfaffenwaldring 57, 70569 Stuttgart, Germany}
\email{marczire@mathematik.uni-stuttgart.de}

\begin{abstract}
We give new properties of algebras with finite Gorenstein dimension coinciding with the dominant dimension $\geq 2$, which are called Auslander-Gorenstein algebras in the recent work of Iyama and Solberg, see \cite{IyaSol}. In particular, when those algebras are standardly stratified, we give criteria when the category of (properly) (co)standardly filtered modules has a nice homological description using tools from the theory of dominant dimensions and Gorenstein homological algebra. We give some examples of standardly stratified algebras having dominant dimension equal to the Gorenstein dimension, including examples having an arbitrary natural number as Gorenstein dimension and blocks of finite representation-type of Schur algebras.
\end{abstract}

\maketitle
\section*{Introduction}
Following the preprint \cite{IyaSol}, we define \emph{Auslander-Gorenstein algebras} to be algebras with Gorenstein dimension $g$ equal to the dominant dimension $d \geq 2$. Those algebras generalize higher Auslander algebras in the sense of Iyama, see \cite{Iya} and algebras of this kind were first considered by Auslander and Solberg in \cite{AS}. 
We show that the Gorenstein homological algebra of Auslander-Gorenstein algebras is directely related to the theory of dominant dimension of modules. Let $Dom_i$ ($Codom_i$) denote the full subcategory of modules having dominant dimension (codominant dimension) at least $i$ and $GProj_j$ ($GInj_j$) the full subcategory of modules having Gorenstein projective (Gorenstein injective) dimension at least $j$.
\begin{theorem }
(see \ref{domgorproequ})
Let $A$ be a Auslander-Gorenstein algebra of Gorenstein dimension $r$. 
\begin{enumerate}
\item $Dom_{r-j}=GProj_j$ for all $j=0,1,..,r$.
\item $Codom_{r-j}=GInj_j$ for all $j=0,1,..,r$.
\end{enumerate}
\end{theorem }
The previous theorem can be seen as a generalisation of theorem 4.7. in \cite{IyaSol} in our setting, as explained in the main text.
Restricting furthermore to Auslander-Gorenstein algebras that are standardly stratified, we describe when the characteristic tilting module has an especially nice form and how the associated categories of modules filtered by standard or proper standard modules look like. 
In a special case, we can give a complete homological description.
This gives a new connection for a large class of algebras between the three subjects of representation theory of finite dimensional algebras consisting of standardly stratified algebras, dominant dimension and Gorenstein homological algebra.
We refer to the main text for the relevant definitions.
\begin{maintheorem}
(see \ref{dualityhigh})
Assume that $A$ is a properly stratified Auslander-Gorenstein algebra with a simple-preserving duality such that the characteristic tilting module coincides with the characteristic cotilting module. Then the Gorenstein dimension of $A$ equals $2m$ for some $m \geq 1$, when the characteristic tilting module has projective dimension $m$ and $F(\bar{\Delta})= Dom_{m}=GProj_m, F(\Delta)=Proj_m,F(\bar{\nabla})= Codom_{m}=GInj_{m}, F(\nabla)=Inj_m$.
\end{maintheorem}
We mention that algebras as in the previous theorem include quasi-hereditary algebras such as representation-finite blocks of Schur algebras and the Auslander algebras of the algebras $K[x]/(x^n)$ for some $n \geq 2$.
Note that the previous theorem applies to all quasi-hereditary higher Auslander algebras with a simple-preserving duality, since in this case the characteristic tilting module automatically coincides with the characteristic cotilting module.
We give examples of Auslander-Gorenstein algebras of arbitrary Gorenstein dimension $g \geq 2$ that are standardly stratified and also show that in general not every higher Auslander algebra is quasi-hereditary. Our methods also apply to give a formula for the relative Auslander-Reiten translate in the category $F(\overline{\Delta})$ for algebras as in the previous theorem.
In \cite{FanKoe3}, the authors defined algebras to be in class $\mathcal{A}$ iff they are quasi-hereditary with a simple-preserving duality and dominant dimension at least two. In the last section we prove the following:
\begin{theorem } (see \ref{classiA} and \ref{schuralgex}.)
Let $K$ be an algebraically closed field and $A$ a basic, representation-finite algebra over $K$, which is in class $\mathcal{A}$. Then $A$ is isomorphic to the Auslander algebra of $K[x]/(x^3)$ or to a representation-finite block of a Schur algebra and has global dimension equal to its dominant dimension equal to $2n-2$, when $A$ has $n \geq 2$ simple modules. Furthermore, one has $F(\Delta)=Dom_{n-1}$ for such algebras with dominant dimension $2n-2$.
\end{theorem }

I thank Volodymyr Mazorchuk for providing a proof of \ref{mazorchuktheo}, {{\O}}yvind Solberg for sending a preprint of \cite{IyaSol} and Steffen Koenig for useful comments.
\section{Preliminaries}
Let an algebra always be a finite dimensional, connected and non-semisimple algebra over a field $K$ and a module over such an algebra is always a finite dimensional right module, unless otherwise stated. $D=Hom_K(-,K)$ denotes the $K$-vector space duality for a given finite dimensional algebra $A$ and $J$ the Jacobson radical. A module is called \emph{basic}, in case it has no direct summand of the form $N^2$ for some indecomposable module $N$. The \emph{basic version} of a module $M$ is a basic module $L$ with $add(L)=add(M)$. 
We define the \emph{dominant dimension} domdim($M$) of a module $M$ with a minimal injective resolution $(I_i): 0 \rightarrow M \rightarrow I_0 \rightarrow I_1 \rightarrow ...$  as:
domdim($M$):=$\sup \{ n | I_i $ is projective for $i=0,1,...,n \}$+1, if $I_0$ is projective, and domdim($M$):=0, if $I_0$ is not projective. \newline
The \emph{codominant dimension} of a module $M$ is defined as the dominant dimension of the module $D(M)$. The dominant dimension of a finite dimensional algebra is defined as the dominant dimension of the regular module $A_A$ and the codominant dimension is the codominant dimension of the module $D(_AA)$.
For $1 \leq i \leq \infty$ $Dom_i(A)$ (resp. $Codom_i(A)$) denotes the full subcategory of mod-$A$ constisting of modules with dominant dimension (resp. codominant dimension) larger than or equal to $i$. For a subcategory $C$ of $mod-A$ and $k \in \{1,2,..., \infty \}$ we denote by $C^{\perp k}$ the subcategory $\{ Y \in mod-A | Ext^{i}(C,Y)=0$ for $i=1,2,...,k \}$ and write $X^{\perp k}$ instead of $add(X)^{\perp k}$ in case $C=add(X)$. $^{\perp k}C$ is defined similarly and we often write  $C^{\perp}$ instead of $C^{\perp \infty}$.
We will use the following theorem of Mueller to calculate dominant dimensions, see \cite{Mue}:
\begin{theorem}\label{Mueller}
Let $A$ be an algebra and $M$ a generator-cogenerator of mod-$A$. Then the dominant dimension of the algebra $B=End_A(M)$ equals $domdim(B)= \inf \{ i \geq 1 | Ext^{i}(M,M) \neq 0 \} +1$.
\end{theorem}
For the calculation of Ext-groups involving simple modules, the following lemma will be useful:
\begin{lemma}\label{benson}
Let $A$ be a finite dimensional algebra, $N$ be an indecomposable $A$-module and $S$ a simple $A$-module.
Let $(P_i)$ be a minimal projective resolution of $N$ and $(I_i)$ a minimal injective resolution of $N$. \newline
1.For $l \geq 0$, $Ext^{l}(N,S) \neq 0$ iff $S$ is a quotient of $P_l$. \newline
2.For $l \geq 0$, $Ext^{l}(S,N) \neq 0$ iff $S$ is a submodule of $I_l$.
\end{lemma}
\begin{proof}
See \cite{Ben} Corollary 2.5.4. 
\end{proof}
An algebra $B$ is called \emph{gendo-symmetric} (see \cite{FanKoe} and \cite{FanKoe2}) in case $B \cong End_A(M)$ for a symmetric algebra $A$ and a generator-cogenerator $M$. In this case one has $A=eBe$, where $eB$ is the minimal faithful projective-injective right $B$-module for some idempotent $e$ and $M \cong Be$ as right $eBe$-modules. Examples of non-symmetric gendo-symmetric algebras are Schur algebras $S(n,r)$ for $n \geq r$, see \cite{FanKoe}.

For the basic notions of the representation theory of finite dimensional algebras including Auslander-Reiten theory, we refer to \cite{SkoYam}.
In the following we always assume that all algebras have dominant dimension larger than or equal to 1. Recall that the dominant dimension of $A_A$ is always equal to the dominant dimension of $_AA$ and thus the dominant dimension of $A$ always equals the codominant dimension (see \cite{Yam}). $eA$ denotes the minimal faithful injective-projective right $A$-module and $Af$ denotes the minimal faithful injective-projective left $A$-module for some idempotents $e$ and $f$ (recall that the existence of such modules is equivalent to $A$ having dominant dimension at least one).
An algebra $A$ is called \emph{$g$-Gorenstein} for a natural number $g$, in case the left and right injective dimensions of $A$ coincide and are equal to $g$. We call $A$ Gorenstein, if $A$ is $g$-Gorenstein for some $g$. In this case $Gordim(A):=injdim(A)$ is called the Gorenstein dimension of $A$.
An $A$-module $M$ is called \emph{Gorenstein projective}, in case $Ext^{i}(M,A)=0$ and $Ext^{i}(D(A),\tau(M))=0$ for every $i \geq 1$. In case $A$ is Gorenstein, being Gorenstein projective is equivalent to $Ext^{i}(M,A)=0$ for every $i \geq 1$, see \cite{Che}. Note that the Gorenstein projective modules coincide with the projective modules for algebras with finite global dimension. The \emph{Gorenstein projective dimension} $GPdim(M)$ of a module $M$ is defined as the minimal $n$ such that there exists an exact minimal sequence (called the minimal Gorenstein projective resolution of $M$):
$0 \rightarrow P_n \rightarrow ... \rightarrow P_1 \rightarrow P_0 \rightarrow M \rightarrow 0$, with the $P_i$ Gorenstein-projective. Dually, one can define Gorenstein injective modules and Gorenstein injective dimensions, see \cite{Che} for more details.
Let $GProj_m(A)$ denote the full subcategory of finitely generated modules of Gorenstein-projective dimension at most $m$ and $GInj_m(A)$ the full subcategory of finitely generated modules with Gorenstein-injective dimension at most $m$. We will write $GProj$ instead of $GProj_0$. Let $Proj_m(A)$ denote the full subcategory of finitely generated modules of projective dimension at most $m$ and $Inj_m(A)$ the full subcategory of finitely generated modules with injective dimension at most $m$. We often just write $Proj_m,Dom_m,...$ for $Proj_m(A),Dom_m(A),...$ if it is clear over which algebra we work.  $\Omega^{i}(A-mod)$ denotes the full subcategory of all projective modules or modules isomorphic to some $\Omega^{i}(N)$ for some $A$-module $N$.
We assume the reader to be familiar with the theory of quasi-hereditary algebras and standardly stratified algebras.
See \cite{DR} for an introduction to quasi-hereditary algebras and \cite{Rei}, \cite{ADL}, \cite{FrMa}, for the basics of standardly stratified algebras. We just briefly recall the most important definitions. 
Let $(A,E)$ be an algebra together with an ordered complete sequence of primitive orthogonal idempotents $E=(e_1,e_2,...,e_n)$. Then the sequence of standard right $A$-modules is defined by $\Delta=(\Delta(1),...,\Delta(n))$, where $\Delta(i)=e_iA/e_iJ(e_{i+1}+e_{i+1}+...+e_n)A$ for $1 \leq i \leq n$.
The sequence of proper standard right $A$-modules $\overline{\Delta}$ is defined by $\overline{\Delta}=(\overline{\Delta}(1),...,\overline{\Delta}(n))$, where $\overline{\Delta}(i)=e_iA/e_iJ(e_i+e_{i+1}+...+e_n)A$ for $1 \leq i \leq n$. Dually, one can define left standard $\Delta^{o}(i)$ and left proper standard modules $\overline{\Delta^{o}}(i)$. The costandard $\nabla(i)$ and proper costandard modules $\overline{\nabla(i)}$ are then defined as the modules $D(\Delta^{o}(i))$ and $D(\overline{\Delta^{o}}(i))$. For a subcategory $C$, let $F(C)$ be the full subcategory of all modules $M$ with a filtration $0 \subseteq M_d \subseteq ... \subseteq M_1=M$, such that every subquotient is isomorphic to an object in $C$. A module is called \emph{standardly filtered} in case $M \in F(\Delta)$ and we use obvious other names for modules in $F(\overline{\Delta})$ etc. $A$ is called \emph{standardly stratified} in case $A \in F(\overline{\Delta})$ and $A$ is called \emph{quasi-hereditary}, in case it is standardly stratified and has finite global dimension. In case $A$ and $A^{op}$ are standardly stratified, $A$ is called \emph{properly stratified}.
In case $A$ is standardly stratified there is a basic tilting module $T$, called the \emph{characteristic tilting module}, uniquely determined by the condition $add(T)=F(\Delta) \cap F(\overline{\nabla})$. In case $A$ is even properly stratified, there is dually a characteristic cotilting module. It is known that for quasi-hereditary algebra, the opposite algebra is again quasi-hereditary and the characteristic tilting module coincides with the characteristic cotilting module.
The class $\mathcal{A}$, introduced by Koenig and Fang in \cite{FanKoe3}, consists of the quasi-hereditary algebras $A$ with dominant dimension at least two and having a simple-preserving duality. For the rest of this article we always assume that dualities preserve simples and we will often just say duality for short. It was shown in \cite{FanKoe3} that all algebras in class $\mathcal{A}$ are gendo-symmetric.
We call an algebra with finite global dimension equal to its dominant dimension $d \geq 2$ a \emph{higher Auslander algebra}, following \cite{Iya}. Auslander algebras are those higher Auslander algebras with global dimension two. We call an algebra with finite Gorenstein dimension equal to its dominant dimension $d \geq 2$ an \emph{Auslander-Gorenstein algebra}, following \cite{IyaSol}.
The finitistic dimension of an algebra is defined as the supremum of the projective dimensions of all modules having finite projective dimension. In case the algebra is Gorenstein, the finititistic dimension coincides with the injective dimension of the regular module.

The following lemma is corollary 3.2. (2) from \cite{CheXi}.
\begin{lemma}
\label{xidom}
Let $0 \rightarrow Y_{-1} \rightarrow Y_0 \rightarrow Y_1 \rightarrow ... \rightarrow Y_{m-1} \rightarrow Y_{m} \rightarrow 0$ be a long exact sequence. Then $domdim(Y_m) \geq \min \{ domdim(Y_j)+j | -1 \leq j \leq m-1 \} -m+1$.
\end{lemma}

Note that for Gorenstein algebras, tilting modules coincide with cotilting modules by the following theorem:

\begin{theorem}\label{gorcrit}
The following are equivalent for an algebra $A$:
\begin{enumerate}
\item $A$ is Gorenstein.
\item There exists a tilting module that is a cotilting module.
\item Every tilting module is a cotilting module.
\end{enumerate}
\end{theorem}
\begin{proof}
See \cite{HU}, Lemma 1.3. 
\end{proof}

The next theorem is the main result in \cite{MazOv}.
\begin{theorem}\label{MazOvtheo}
Let $A$ be a properly stratified algebra with a duality and such that the characteristic tilting $T$ module coincides with the characteristic cotilting module. Then the finitistic dimension of $A$ equals two times the projective dimension of $T$.
\end{theorem}
Note that dually the finitistic dimension of $A$ equals two times the injective dimension of $T$ and that an algebra as in the previous theorem has to be Gorenstein by \ref{gorcrit}.

\begin{theorem}
\label{gorprores}
Let $M$ be an $A$-module. Then the following are equivalent:
\begin{enumerate}
\item $GPdim(M) \leq n$
\item There exists an exact sequence $0 \rightarrow G_n \rightarrow ... \rightarrow G_1 \rightarrow G_0 \rightarrow M  \rightarrow 0$ with each $G_i$ Gorenstein projective.
\item For each exact sequence $0 \rightarrow K \rightarrow G_{n-1} \rightarrow ... \rightarrow G_1 \rightarrow G_0 \rightarrow M \rightarrow 0$ with each $G_i$ Gorenstein projective, also $K$ is Gorenstein projective.
\end{enumerate}
\end{theorem}
\begin{proof}
See \cite{Che} Theorem 3.2.5.
\end{proof}

Following \cite{Che}, define the \emph{global Gorenstein projective dimension} of $A$ as $glGPdim(A):=sup \{ GPdim(M) | M \in mod-A \}$.
\begin{theorem}
\label{gorensteinchara}
For an algebra $A$ the following are equivalent:
\begin{enumerate}
\item $A$ is Gorenstein of Gorenstein dimension $n$.
\item $A$ has finite global Gorenstein projective dimension $n$.
\item $GProj=\Omega^{n}(A-mod)$.
In case $A$ is Gorenstein, the Gorenstein dimension equals the finitistic dimension of $A$.
\end{enumerate}
\end{theorem}
\begin{proof}
See \cite{Che}, Theorem 2.3.7.,Corollary 3.2.6. and Theorem 3.2.7. 
\end{proof}

\begin{theorem}\label{marviltheo}
\label{marvil}
Let $A$ be an algebra of dominant dimension $n \geq 1$. Let $i \leq n$, then the following holds:
\begin{enumerate}
\item $Dom_i = \Omega^{i}(A-mod)$.
\item $Dom_i^{\perp}=Inj_i , ^{\perp}Inj_i=Dom_i, Proj_i^{\perp}=Codom_i, ^{\perp}Codom_i=Proj_i$.
\end{enumerate}
\end{theorem}
\begin{proof}
See \cite{MarVil}, Proposition 4 and the corollary after it. 
\end{proof}

Tilting modules here always denote tilting modules of arbitrary finite projective dimension. For a subcategory $C$ of $mod-A$, we denote by $\overline{C}$ the subcategory of $mod-A$ that consists of all module $M$ for which an exact sequence $0 \rightarrow C_n \rightarrow ... \rightarrow C_1 \rightarrow C_0 \rightarrow M \rightarrow 0$ exists with $C_i \in C$. Dually $\underline{C}$ is the subcategory of $mod-A$ consisting of all modules $M$ for which an exact sequence $0 \rightarrow M \rightarrow C_0 \rightarrow C_1 \rightarrow ... \rightarrow C_n \rightarrow 0$ exists with $C_i \in C$. 
\begin{proposition}
\label{reiprop}
\begin{enumerate}
\item Let $T$ be a tilting module. Then $^{\perp}(T^{\perp})=\underline{addT}$.
\item Let $U$ be a cotilting module. Then $(^{\perp}U)^{\perp}=\overline{add(U)}$.
\end{enumerate}
\end{proposition}
\begin{proof}
See \cite{Rei}, Proposition 1.17. and below that proposition.
\end{proof}

We will use the following construction of \cite{CheKoe} corollary 5.4., to get examples of gendo-symmetric algebras with dominant dimension equal to their Gorenstein dimension:
\begin{proposition}\label{chekoeres}
Let $A$ be a symmetric algebra and $M=A \oplus N$ a generator of $mod-A$ such that $n+1=\inf \{ i \geq 1 | Ext^{i}(M,M) \neq 0 \}< \infty$ and $\Omega^{n+2}(N)=N$. Then $End_A(M)$ has dominant dimension equal to Gorenstein dimension equal to $n+2$.  
\end{proposition}
Note that the previous proposition is always satisfied in case $M=A \oplus N$, where $N$ is a 2-periodic module.

\section{Auslander-Gorenstein algebras, which are standardly stratified}
Assume always, that $A$ is an algebra with dominant dimension at least one. Then there is a minimal faithful injective-projective right module $eA$ and a minimal faithful injective-projective left module $Af$ and we fix this notation and idempotents $e$ and $f$. We also fix $t$ for the number of simple modules in $A$. Since we are interested mainly in homological questions, we can assume that all our modules and algebras are basic.

\begin{lemma}
Let $A$ be a higher Auslander-Gorenstein algebra of Gorenstein dimension $r$. Then $GProj=Dom_r$.
\end{lemma}
\begin{proof}
By definition, $A$ has dominant dimension $r$.
Use \ref{marvil} to see that for $i \leq r$, $Dom_i=\Omega^{i}(mod-A)$. On the other hand an algebra $A$ is Gorenstein iff $GProj=\Omega^{r}(mod-A)$ by \ref{gorensteinchara}. Combining those two results, the lemma follows.
\end{proof}

\begin{theorem}
\label{domgorproequ}
Let $A$ be a Auslander-Gorenstein algebra of Gorenstein dimension $r$. 
\begin{enumerate}
\item $Dom_{r-j}=GProj_j$ for all $j=0,1,..,r$.
\item $Codom_{r-j}=GInj_j$ for all $j=0,1,..,r$.
\end{enumerate}
\end{theorem}
\begin{proof}
We only prove 1, since the proof of 2. is dual. \newline 
The case $j=0$ was shown in the previous lemma.
Assume $M \in Dom_{r-j}$ for $j \geq 1$ and assume $M$ is not Gorenstein-projective (if $M$ is Gorenstein projective, $M \in GProj_j$ is clear). Look at the beginning of a minimal injective resolution of $M$: \newline
$0 \rightarrow M \rightarrow I_0 \rightarrow ... \rightarrow I_{r-j-1} \rightarrow K \rightarrow 0$ 
and assume that $M$ has Gorenstein projective dimension strictly larger than $j$. Then $K$ has Gorenstein projective dimension at least $r+1$ as the following sequence shows, where the $P_i$ come from a minimal Gorenstein-projective resolution of $M$: \newline
$... \rightarrow P_{j+1} \rightarrow P_{j} \rightarrow ... \rightarrow P_0 \rightarrow I_0 \rightarrow ... \rightarrow I_{(r-j)-1} \rightarrow K \rightarrow 0.$ \newline
Here the map $P_0 \rightarrow I_0$ is the composition of $P_0 \rightarrow M$ and $M \rightarrow I_0$.
But this contradicts the fact that $A$ has finite global Gorenstein projective dimension $r$, by \hyperref[gorensteinchara]{ \ref*{gorensteinchara}}, and so $M$ has Gorenstein projective dimension at most $j$.
Now assume that $M \in GProj_j$. Then take the minimal Gorenstein-projective resolution of $M$:
$0 \rightarrow P_j \rightarrow P_{j-1} \rightarrow ... \rightarrow P_0 \rightarrow M \rightarrow 0$.
By \hyperref[xidom]{ \ref*{xidom}}, $domdim(M) \geq r-1-j+1=r-j$, since every Gorenstein projective module has dominant dimension at least $r$, by the previous lemma.
Thus $M \in Dom_{r-j}$. 
\end{proof}

In our setting the previous theorem can be seen as a generalisation of theorem 4.7. of \cite{IyaSol}, which corresponds to the case j=0, by noting that the category $M^{l-2}$ is equivalent to $Dom_l$ via the functor $Hom_B(M,-)$ for any $l \geq 2$, when $A \cong End_B(M)$ for some algebra $B$ with generator-cogenerator $M$, see for example \cite{APT} proposition 3.7.

\begin{proposition}\label{projinjdimcrit}
Let $A$ be a Auslander-Gorenstein algebra of Gorenstein dimension $r$ and let $X$ be a module of projective dimension $i$ and injective dimension $r-i$. Then $X$ has exactly dominant dimension $r-i$ and codominant dimension $i$ and is contained in $add(\Omega^{r-i}(D(A)))=add(\Omega^{-i}(A))$. Conversly, every module in $add(\Omega^{r-i}(D(A)))=add(\Omega^{-i}(A))$ has this property.

\end{proposition}
\begin{proof}
Since $X$ having projective dimension $i$ implies that $X \in GProj_i$, by the previous theorem $X$ has dominant dimension at least $r-i$. But the dominant dimension is always smaller than or equal to the positive injective dimension and so the dominant dimension of $X$ equals $r-i$. Dually $X$ has codominant dimension $i$. There exists the following exact sequence coming from a minimal projective resolution $(P_i)_i$ of $X$ and a minimal injective resolution $(I_i)_i$ of $X$: \newline
$0 \rightarrow P_i \rightarrow ... \rightarrow P_0 \rightarrow X \rightarrow I_0 \rightarrow ... \rightarrow I_{r-i-1} \rightarrow I_{r-i} \rightarrow 0$. The module $I_{r-i}$ is not projective, since the dominant dimension of $X$ is exactly $r-i$. Thus $I_{r-i}$ has a nonprojective injective indecomposable summand and $X \in add(\Omega^{r-i}(D(A)))$ is clear, since $I_0, I_1 ,..., I_{r-i-1}$ are projective. The equality $add(\Omega^{r-i}(D(A)))=add(\Omega^{-i}(A))$ follows by looking at the following  injective resolution of $A$, which has the form $0 \rightarrow A \rightarrow I_0 \rightarrow ... \rightarrow I_{r-1} \rightarrow D(A) \rightarrow 0$ and all terms $I_i$ are projective, since $A$ is a Auslander-Gorenstein algebra. That every module in $add(\Omega^{r-i}(D(A)))=add(\Omega^{-i}(A))$ has projective dimension $i$ and injective dimension $r-i$ is obvious.
\end{proof}

\begin{corollary}
\label{tiltchar}
Let $A$ be a Auslander-Gorenstein algebra of Gorenstein dimension $r$. Let $T$ be a basic tilting-cotilting module of projective dimension $i \geq 1$ and of injective dimension $r-i$. Then $T$ has dominant dimension $r-i$ and codominant dimension $i$. Furthermore $T \cong eA \oplus \Omega^{r-i}(D(A)) \cong eA \oplus \Omega^{-i}(A)$.
Thus $T$ is the unique basic tilting-cotilting module in the categories $Dom_{r-i} \cap Codom_i$, $Proj_i \cap Codom_i$ and $Inj_{r-i} \cap Dom_{r-i}$.
\end{corollary}
\begin{proof}
The corollary follows from the previous proposition and noting that the number of indecomposable summands of a basic tilting module is always equal to the number of simple $A$-modules, which is the same as the number of indecomposable summands of the module $eA \oplus \Omega^{r-i}(D(A)) \cong eA \oplus \Omega^{-i}(A)$.
\end{proof}
\begin{proposition}
\label{tiltorth}
Let $A$ be a Auslander-Gorenstein algebra of Gorenstein dimension $r$ and $T$ a basic tilting-colting module of projective dimension $i \geq 1$ and injective dimension $r-i$.
Then $^{\perp}T  = GProj_i = Dom_{r-i}, \underline{add(T)}=^{\perp}(T^{\perp})=Proj_i , \overline{add(T)}=(^{\perp}(T))^{\perp}=Inj_{r-i}$ and $T^{\perp} = GInj_{r-i}=Codom_i$.
\end{proposition}
\begin{proof}
We use that  $T \cong eA \oplus \Omega^{r-i}(D(A)) \cong eA \oplus \Omega^{-i}(A)$, as was shown in the previous corollary.
$^{\perp}T = \{ M | Ext^{l}(M,T)=0$ for every $l \geq 1 \} = \{ M | Ext^{l}(M,\Omega^{-i}((1-e)A))=0$ for every $l \geq 1 \} =  \{ M | Ext^{l+i}(M,(1-e)A)=Ext^{l}(\Omega^{i}(M),(1-e)A)=0$ for every $l \geq 1 \}$. Now note that any module $Y$ with $Ext^{l}(Y,A)=0$ for every $l \geq 1$ is Gorenstein projective since $A$ has finite Gorenstein dimension. Thus  $\{ M | Ext^{l}(\Omega^{i}(M),(1-e)A)=Ext^{l}(\Omega^{i}(M),A) =0$ for every $l \geq 1 \}= \{ M | \Omega^{i}(M)$ is Gorenstein-projective $ \} = GProj_i$. The equality $GProj_i = Dom_{r-i}$ follows from \hyperref[domgorproequ]{ \ref*{domgorproequ}}.
Note that we used $Ext^{i}(\Omega^{m}(M),(1-e)A)=Ext^{i}(\Omega^{m}(M),A) =0$, since $eA \oplus (1-e)A=A$ and $eA$ is projective and injective, so Ext vanishes for this summand in all terms. 
$\underline{add(T)}=^{\perp}(T^{\perp})$ follows using \hyperref[reiprop]{ \ref*{reiprop}}. $T^{\perp}=Codom_{i}$ follows dually. Now note that $^{\perp}(T^{\perp})=^{\perp}(Codom_{i})=Proj_i$ by \hyperref[marvil]{ \ref*{marvil}}.
$T^{\perp }=GInj_{r-i}$ and $\overline{add(T)}=(^{\perp}(T))^{\perp}=Inj_{r-i}$ are proved dually.

\end{proof}

\begin{theorem}\label{maintheorem}
Let $A$ be a Auslander-Gorenstein algebra of Gorenstein dimension $r$.
Assume furthermore that $A$ is standardly stratified with characteristic tilting module $T$.
Then the following are equivalent:
\begin{enumerate}
\item $T:=eA \oplus \Omega^{-i}((1-e)A)$ is the characteristic tilting module of the stratified structure on $A$.
\item Every standard module $\Delta(i)$ has dominant dimension at least $r-i$ and every proper costandard module $\bar{\nabla}(i)$ has codominant dimension at least $i$.
\item $F(\Delta) \subseteq Dom_{r-i}=GProj_i$ and $F(\bar{\nabla}) \subseteq Codom_i=GInj_{r-i}$.
\item $Proj_i = F(\Delta)$ and $Codom_{i}= GInj_{r-i} = F(\bar{\nabla})$.
\end{enumerate}
\end{theorem}
\begin{proof} \ \
\begin{enumerate} 
\item[ \ 1. $\Rightarrow$ 2. ] and 1. $\Rightarrow$ 4.: Assume first that $T$ is the characteristic tilting module of the stratified structure on $A$. Using \hyperref[tiltorth]{ \ref*{tiltorth}}, $F(\Delta)=\underline{add(T)}=Proj_{i} \subseteq GProj_i=Dom_{r-i}$ and $F(\bar{\nabla})=T^{\perp} = GInj_{r-i}=Codom_i$
\item[ \ 2. $\Rightarrow$ 3.:] Every module in $F(\Delta)$ has dominant dimension at least $r-i$, since every module in $F(\Delta)$ is build from extensions of the modules $\Delta(x)$, which have dominant dimension at least $r-i$. $F(\bar{\nabla}) \subseteq Codom_i$ has an analog proof.
\item[ \ 3. $\Rightarrow$ 1.:] Then $T \in add(F(\Delta) \cap F(\bar{\nabla}))$ by the definition of the characteristic tilting module and by 3.:  $add(F(\Delta) \cap F(\bar{\nabla})) \subseteq add(Dom_{r-i} \cap Codom_i)$. But by  \hyperref[tiltchar]{ \ref*{tiltchar}}, the only basic tilting module in $add(Dom_{r-i} \cap Codom_i)$ is $T= eA \oplus \Omega^{-m}((1-e)A)$. 
\item[ \ 4. $\Rightarrow$ 3.:] This is trivial. 
\end{enumerate}
\end{proof}

\begin{theorem}
\label{dualityhigh}
Assume that $A$ is a properly stratified Auslander-Gorenstein algebra with a duality such that the characteristic tilting module coincides with the characteristic cotilting module. Then the Gorenstein dimension of $A$ equals $2m$ for some $m \geq 1$, when the characteristic tilting module has projective dimension $m$ and thus $F(\bar{\Delta})= Dom_{m}=GProj_m, F(\Delta)=Proj_m,F(\bar{\nabla})= Codom_{m}=GInj_{m}, F(\nabla)=Inj_m$.
\end{theorem}
\begin{proof}
Since $A$ has finite Gorenstein dimension, the Gorenstein dimension coincides with the finitistic dimension of $A$. We can apply theorem \ref{MazOvtheo}, and get $2projdim(T)=2m=2injdim(T)$, where $2m$ is the Gorenstein dimension of $A$. Thus by \ref{projinjdimcrit}, $T$ has to be isomorphic to $eA \oplus \Omega^{-m}((1-e)A)$ and the result follows from the theorem \ref{maintheorem}.
\end{proof}
We note the following special case for quasi-hereditary algebras explicitly:
\begin{corollary}\label{quasiheredduality}
Assume that $A$ is a quasi-hereditary higher Auslander algebra with a duality. Then the global dimension of $A$ equals $2m$ for some $m \geq 1$ and the characteristic tilting module has projective dimension $m$ and thus $F(\Delta)= Dom_{m}=Proj_m, F(\nabla)=Codom_m=Inj_m$.
\end{corollary}

\begin{remark}\label{mazorchuktheo}
The author conjectured the following: \newline
Let $A$ be a properly stratified algebra. Then the following are equivalent:
\begin{enumerate}
\item $A$ is Gorenstein.
\item The characteristic tilting module coincides with the characteristic cotilting module.
\end{enumerate}
Volodymyr Mazorchuk could sketch a proof of the theorem. We will provide a proof in forthcoming work and note that the theorem implies that the assumption that the characteristic cotilting module coincides with the characteristic tilting module in \ref{dualityhigh} is not necessary.
\end{remark}

\section{examples and applications}
\subsection{Examples and counterexamples}
We assume here that all Nakayama algebras are given by quiver and relations and their simple modules are numbered from $0$ to $n-1$, when the algebra has $n$ simple modules. In case their quiver is linear, the leftmost point has index $0$ and the rightmost has index $n-1$ and the arrows of the quiver point from left to right.  The Kupisch series $[a_0,a_1,...,a_{n-1}]$ just gives the dimension of the indecomposable projective modules $e_iA$ at point $i$. See \cite{ARS} and \cite{Mar} for more details on Nakayama algebras.
Our first example shows that there exist quasi-hereditary higher Auslander algebras of arbitrary global dimension $\geq 2$.
\begin{example}
Let $A$ be the Nakayama algebra with $n$ simple modules and Kupisch series $[2,2,...,2,3]$. Then $A$ is a higher Auslander algebra with global dimension equal to $n$. It is quasi-hereditary with ordering $[1,2,3,...,n-1,0]$ and $F(\Delta)=Dom_1, F(\nabla)=Codom_{n-1}$, since all modules $\Delta(i)$ have dominant dimension at least one and all costandard modules have codominant dimension at least $n-1$, using \ref{maintheorem}.
\end{example}

The next two algebras show that in contrast to Auslander algebras, higher Auslander algebras or Auslander-Gorenstein algebras of Gorenstein dimension two are in general not standardly stratified.
\begin{example}
Let $A$ be the Nakayama algebra with Kupisch series $[4,5,5]$, then $A$ is a Auslander-Gorenstein algebra with Gorenstein dimension two. But $A$ is not standardly stratified.
\end{example}

\begin{example}
The Nakayama algebra $A$ with Kupisch series $[3,4,4]$ has global dimension equal to its dominant dimension, which both are equal to 4. But $A$ is not quasi-hereditary.
This shows that not every higher Auslander algebra is quasi-hereditary in contrast to the case of Auslander algebras.
\end{example}

The following main result from \cite{CheDl} provides many examples of Auslander-Gorenstein algebras with Gorenstein dimension two, which are also properly stratified with a duality. 
\begin{theorem}
\label{chedlabtheo}
\begin{enumerate}
\item Let $A$ be a local, commutative selfinjective algebra over an algebraically closed field. Let $\mathcal{X}=(A=X(1),X(2),...,X(n))$ be a sequence of local-colocal modules (meaning that all modules have simple socle and top and therefore can be viewed as ideals of $A$) with $X(i) \subseteq X(j)$ implying $j \leq i$. Let $X= \bigoplus\limits_{i=1}^{n}{X(i)}$ and $B=End_A(X)$. Then $B$ is properly stratified with a duality iff the following two conditions are satisfied:  \newline
1. $X(i) \cap X(j)$ is generated by suitable $X(t)$ of $\mathcal{X}$ for any $1 \leq i,j \leq n$ \newline
2. $X(j) \cap \sum\limits_{t=j+1}^{n}{X(t)}=\sum\limits_{t=j+1}^{n}{X(j) \cap X(t)}$ for any $1 \leq j \leq n$.
\item Such an algebra $B$ has Gorenstein dimension equal to dominant dimension equal to 2 iff $\bigoplus\limits_{i=2}^{n}{X(i)}$ is 2-periodic.
\end{enumerate}
\end{theorem}
\begin{proof}
\begin{enumerate}
\item This is theorem 2.4. in \cite{CheDl}.
\item This follows from \ref{chekoeres}.
\end{enumerate}
\end{proof}

\begin{example}
Let $Q$ be a quiver with 1 point and two loops $x$ and $y$ and $A=KQ/I$, where \newline $I=<x^2,y^2,xy-yx>$. Note that $A$ is symmetric and isomorphic to the group algebra of the Klein four group in case $K$ has characteristic 2. Let $M:=A \oplus xA$. Then $B:=End_A(M)$ is properly stratified, has a duality and has dominant dimension equal to Gorenstein dimension equal to two since $xA$ is 2-periodic, using \ref{chedlabtheo}. This implies $F(\bar{\Delta})=Dom_1$ by \ref{dualityhigh}. $B$ does not have finite global dimension (otherwise $B$ would have global dimension equal to its Gorenstein dimension equal to two, but $B$ is certainly not an Auslander algebra as $add(M)$ is not equal to $mod-A$) and thus is not quasi-hereditary.
\end{example}

\subsection{Applications to relative Auslander-Reiten theory}
We show that our methods allow us to give an explicit form of the relative Auslander-Reiten translates in $F(\overline{\Delta})$ of certain standardly stratified algebras. Before doing so, we recall some results from the literature.
We refer the reader to \cite{Kl} for the standard facts on relative homological algebra needed here. Recall that extension-closed functorially finite subcategories $C$ in $mod-A$ have almost split sequences and we denote the relative Auslander-Reiten translate in that subcategory by $\tau_C$.
The next result can be found in section 3 of \cite{AS2}. 
\begin{theorem}
The category $\Omega^{n}(A-mod)$ is functorially finite and for a given nonprojective module $M$, a right $\Omega^{n}(A-mod)$-approximation of $M$ is of the form $\Omega^{n}(\Omega^{-n}(M)) \oplus P \rightarrow M$, for some projective module $P$.
\end{theorem}
The next theorem is a special case of \cite{Kl}, theorem 2.3.
\begin{theorem}
Let $C$ be an extension-closed functorially finite subcategory and $M$ a module in $C$.
Then $\tau_C(M)$ equals the unique indecomposable summand $Y$ of a right $C$-approximation of $M$ with $Ext^{1}(M,Y) \neq 0$.
\end{theorem}

\begin{proposition}
Assume that $A$ is a properly stratified Auslander-Gorenstein algebra with a simple-preserving duality with dominant dimension $2m$ for some $m \geq 1$ such that the characteristic tilting module coincides with the characteristic cotilting module. Let $C:=F(\hat{\Delta})$ and $M \in C$ a non Ext-projective module. Then $\tau_C(M)$ is the unique indecomposable summand $Y$ of $\Omega^{m}(\Omega^{-m}(\tau(M)))$ with $Ext^{1}(M,Y) \neq 0$.
\end{proposition}
\begin{proof}
This follows from the previous two theorems by using that $F(\overline{\Delta})=Dom_m$ from \ref{dualityhigh} and $Dom_m = \Omega^{m}(A-mod)$, by \ref{marviltheo}. 
\end{proof}

We give one example, where one can explicitly calculate the relative Auslander-Reiten translate using the previous proposition.
\begin{example}\label{nakarq}
Let $A$ be the Nakayama algebra with Kupisch series $[2d,2d+1]$ for some $d \geq 1$. Note that $A$ is isomorphic to $End_B(B \oplus rad(B))$, when $B=K[x]/(x^{d+1})$. Since $rad(B)$ is 2-periodic, we get by \ref{chedlabtheo} that $A$ is a Auslander-Gorenstein algebra with Gorenstein dimension two and $A$ is also properly stratified with a duality by \ref{chedlabtheo}. $A$ has finite global dimension iff $d=1$.
$A$ has a unique indecomposable projective-injective module $e_1A \cong D(Ae_1)$.
By \ref{dualityhigh}, $F(\overline{\Delta})=Dom_1$ and $F(\Delta)=Proj_1$. Let $C=Dom_1$. Note that an indecomposable module of the form $e_iA/e_iJ^k$ has injective envelope $D(Ae_{i+k-1})$ and thus has dominant dimension at least one, iff $i+k-1 \equiv 1$ mod $2$ iff $i \equiv k$ mod $2$.
In case $e_iA/e_iJ^k$ has dominant dimension at least one, one has $\tau(e_iA/e_iJ^k)=e_{i+1}A/e_{i+1}J^k$ , $\Omega^{-1}(e_{i+1}A/e_{i+1}J^k)=D(J^ke_0)=e_1A/e_1J^{2d-k}$ and finally $\tau_C(e_iA/e_iJ^k)=\Omega^{1}(\Omega^{-1}(e_{i+1}A/e_{i+1}J^k))=e_{1-k}A/e_{1-k}J^{1+k}$.
The modules with projective dimension less than or equal to one are exactly the projective modules and the simple modules $S_1$. The corresponding relative Auslander-Reiten sequence for $S_1$ is \newline $0 \rightarrow e_0A/e_0J^2 \rightarrow e_1A/e_1J^3 \rightarrow e_1A/e_1J^1 \rightarrow 0$ in $F(\Delta)$ and $F(\overline{\Delta})$, which in case $d=1$ is the only relative Auslander-Reiten sequence and simplifies to $0 \rightarrow e_0A \rightarrow e_1A \rightarrow e_1A/e_1J^1 \rightarrow 0$ for $d=1$. We now assume that $d \geq 2$, so that $A$ has infinite global dimension and calculate the relevant relative almost split sequences for $F(\overline{\Delta})$ in that case.
The remaining relative Auslander-Reiten sequences are as follows for $k=1,2,...,d-1$: \newline
$0 \rightarrow e_1A/e_1J^{2k+1} \rightarrow e_1A/e_1J^{2k-1} \oplus e_0A/e_0J^{2k+2} \rightarrow e_0A/e_0J^{2k} \rightarrow 0$, \newline
$0 \rightarrow e_0A/e_0J^{2k+2} \rightarrow e_1A/e_1J^{2k+3} \oplus e_0A/e_0J^{2k} \rightarrow e_1A/e_1J^{2k+1} \rightarrow 0$.
\end{example}
\section{Classification of representation-finite algebras of class $\mathcal{A}$}
In this section, we always assume that our field $K$ is algebraically closed. Note that \ref{quasiheredduality} applies to higher Auslander algebras in class $\mathcal{A}$.
We will classify all representation finite algebras in the class $\mathcal{A}$ up to Morita equivalence and show that they are in fact higher Auslander algebras. Thus those examples all have the property that $F(\Delta)=Dom_m=Proj_m$ by \hyperref[dualityhigh]{ \ref*{dualityhigh}}, when the algebra has global dimension $2m$.
Since we are only interested in a classification up to Morita equivalence, we assume in this chapter that all algebras are given by quiver and relations.
We assume that all algebras have at least two simple modules to avoid trivialities (recall that the only quiver algebra with only one simple module and finite global dimension is the field).
The following theorem has been proven by Reiten and Donkin in \cite{DoR}: 
\begin{theorem}\label{donkinreiten}
Let $A$ be a representation finite quasi-hereditary quiver algebra having a duality with $m \geq 2$ simple modules.
Then $A$ is isomorph to one of the following algebras: 
\begin{enumerate}
\item The Auslander algebra of $K[x]/(x^3)$ with 3 simple modules.
\item The algebra $B(m,\lambda_1,\lambda_2,...,\lambda_{n-2})$, with $\lambda_{i-1} \in \{ 0,1 \}$ for $i=2,..,m-1$ given by quiver and relations as follows:
$$
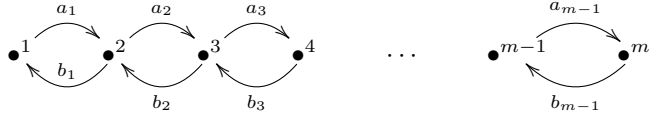
\xymatrix@1{\bullet^1 \ar@/^1pc/ [r]^{a_1} & \bullet^2 \ar@/^1pc/ [r]^{a_2} \ar@/^1pc/ [l]_{b_1}  & \bullet^3 \ar@/^1pc/ [r]^{a_3} \ar@/^1pc/ [l]^{b_2}& \bullet^4 \ar@/^1pc/ [l]^{b_3} & \cdots& \bullet^{m-1} \ar@/^1pc/ [r]^{a_{m-1}} & \bullet^{m} \ar@/^1pc/ [l]^{b_{m-1}}}$$
$b_{m-1} a_{m-1}, b_{i-1}a_{i-1}- \lambda_{i-1} a_i b_i, a_{i-1}a_i, b_i b_{i-1}$.
 
\end{enumerate}
\end{theorem}

The next lemma describes in general how to obtain the quiver with relations of an algebra of the form $End_A(A\oplus \bigoplus\limits_{k=1}^{n}{S_{i_k}})$, where $S_i$ are some simple $A$-modules of a symmetric quiver algebra $A$.
A map of the form $l_{\alpha}$ for a linear combination of paths $\alpha$ always denotes left multiplication by that linear combination.
\begin{lemma}\label{quiverberechnen}
Let $A$ be a symmetric quiver algebra with primitive idempotents $e_i$ such that all indecomposable projective modules have Loewy length at least three (with the primitive idempotents arbitrary ordered) and $soc(e_iA)=e_i J^{n_i}=< \delta_i>$ ($i=1,2,...,n$) its simple modules, where $\delta_i$ denotes a vector space basis of $e_i J^{n_i}$. Then for $p \leq n$ the quiver with relations of the algebra $B:= End_A(A\oplus \bigoplus\limits_{i=1}^{p}{e_i J^{n_i}})$ can be obtained from the quiver with relations of $A$ in the following way:
For every simple module $e_i J^{n_i}$ appearing as a direct summand of $A \oplus \bigoplus\limits_{i=1}^{p}{e_i J^{n_i}}$ do the following: add a new vertex $p_{e_iJ^{n_i}}$ to the quiver of $A$ and two new arrows, called $\alpha_{e_i J^{n_i}}$ (pointing from $e_i$ in $A$ to the vertex $p_{e_i J^{n_i}}$) and $\beta_{e_i J^{n_i}}$ (pointing from the vertex $p_{e_i J^{n_i}}$ to the vertex $e_i$ in $A$) which we call $p_i$, $\alpha_i$ and $\beta_i$ for short assuming no such vertex or arrow already is named like that in $A$. The relations of $B$ are the old relations of $A$ plus the following new relations for every $i$: \newline
$0=  \gamma \alpha_i = \beta_i \gamma$ for any path $\gamma$ and $\alpha_i \beta_i =\delta_i$.
\end{lemma}
\begin{proof}
First we calculate $B$ as a matrix algebra. We use that $Hom_A(e_i J^{n_i},e_j J^{n_j})=0$ for any two nonisomorphic simple $A$-modules. In $Hom_A(e_i J^{n_i},A)$ we choose the inclusion $inc_i$ as a canonical basis element and in $Hom_A(A,e_i J^{n_i}) \cong e_i J^{n_i},$ we choose $l_{\delta_i}$ as a canonical basis element. For $End_A(e_i J^{n_i})$, we choose the identity $id_i$ of $e_i J^{n_i}$ as a basis element.
Then $B$, as a matrix algebra of $(p+1) \times (p+1)$ matrices, is given by 
\[ \left( \begin{array}{ccccc}
A & <inc_1> & <inc_2> & ... & <inc_l> \\
<l_{\delta_1}> & <id_1> & 0 & ... & 0 \\
<l_{\delta_2}> & 0 & <id_2> & ... & 0 \\
... & ... & ... & ... & ...  \\
<l_{\delta_p}> & 0 & 0 & ... & <id_p>  \end{array} \right) . \]
Note that $B$ has dimension $dim(B)=dim(A)+3p$.
Define $p_{\alpha_i}$ as the matrix having $inc_i$ in the entry $(1,i+1)$ and everywhere else zeros.
Define $p_{\beta_i}$ as the matrix having $\delta_i$ in the entry $(i+1,1)$ and everywhere else zeros.
Define $p_{e_iJ^{n_i}}$ as the matrix having $<id_i>$ as the entry $(i+1,i+1)$ and everywhere else zeros.
For every arrow or vertex $x$ in $A$ define the matrix $p_x$ as follows:
$p_x:=$
\[ \left( \begin{array}{ccccc}
x & 0 & 0 & ... & 0 \\
0 & 0 & 0 & ... & 0 \\
0 & 0 & 0 & ... & 0 \\
0 & 0 & 0 & ... & 0  \end{array} \right). \]
Let $KQ$ be the quiver algebra constructed in the lemma and $\pi: KQ \rightarrow B$ the $K$-algebra homomorphism sending the vertices or arrows $y$ of $Q$ to $p_y$ in $B$. By construction, $\pi$ has the relations of $A$ and the new relations  $\gamma \alpha_i =0= \beta_i \gamma=$ for any path $\gamma$ and $\alpha_i \beta_i =\delta_i$ for any $i$ in the kernel $I$ and so it induces a surjective map $\hat{\pi}: KQ/I \rightarrow B$. Since $KQ/I$ has the same dimension $dim(A)+3p$ as $B$, $\hat{\pi}$ is an isomorphism. 
\end{proof}
The assumption in the previous lemma that all indecomposable projective modules have Loewy length at least three, implies that $\delta_i$ is contained in $J^2$ and therefore the above new relations will really be admissible.

\begin{proposition}
The algebra $B(n,\lambda_1,\lambda_2,...,\lambda_{n-2})$ has dominant dimension zero if one $\lambda_i$ is zero.
\end{proposition}
\begin{proof}
Assume that $\lambda_k=0$ for some $k$. We show that $A:=B(n,\lambda_1,\lambda_2,...,\lambda_{n-2})$ has dominant dimension zero. We have $b_k a_k=0$ and $e_{k+1}A=<e_{k+1},a_{k+1},a_{k+1} b_{k+1}, b_k>$. Thus soc($e_{k+1}A$)=$<a_{k+1} b_{k+1},b_k>$ and so the injective hull of $e_{k+1}A$ is $D(Ae_k) \oplus D(A e_{k+1})$. We show that $D(Ae_{k+1})$ is not projective and thus the dominant dimension of $e_{k+1}A$ is zero. We show that the top of the module $D(Ae_{k+1})$ is not simple and thus $D(Ae_{k+1})$ can not be a projective module, since $D(Ae_{k+1})$ is indecomposable.
To calculate the top of the module $D(Ae_{k+1})$, we can equivalently calculate the socle of the module $Ae_{k+1}$, because one has in general $top(D(M)) \cong D(soc(M))$. Now soc($Ae_{k+1}$)=$<a_k,a_{k+1} b_{k+1} >$, which is not simple. Thus $A$ has dominant dimension zero as the dominant dimension of an algebra is the minimum of the dominant dimensions of the indecomposable projective modules.
\end{proof}

\begin{proposition}
The algebra $B(n,1,1,...,1)$ has dominant dimension $2n-2$ and is in class $\mathcal{A}$.
\end{proposition}
\begin{proof}
We assume that $B(n,1,1,...,1)$ has at least 3 simple modules, since for two simple modules the algebra is isomorphic to the Auslander algebra of $K[x]/(x^2)$ and the result is clear in this case.
Now define for $n \geq 2$ the following algebra $A:=kQ/I$ with $n$ simple modules:
The quiver is the same as in \ref{donkinreiten} and the relations are $b_{i-1} a_{i-1} = b_i a_i, a_{i-1} a_i=0, b_i b_{i-1}=0$ for $i=2,3,...,n$ (so we omit one relation). Then it is easily checked that the socle of the module $e_n A$ is equal to $<b_{n-1} a_{n-1}>$ and using the previous lemma \ref{quiverberechnen}, $End_A(A \oplus S_n)$ is isomorphic to $B(n+1,1,1,...,1)$. To show that $B(n+1,1,1,...,1)$ has dominant dimension $2n$, we use Mueller's theorem \ref{Mueller} and check that $Ext_A^{i}(S_n,S_n)=0$ for $i=1,...,2n-2$, but $Ext_A^{2n-1}(S_n,S_n)\neq 0$.
We construct a minimal projective resolution of $S_n$:
$$
\begin{xy}
	\xymatrix@C-1pc@R-1pc{
	& e_1A \ar[rd]^{}   &         & ..... \ar[rd]^{} &     & e_{n-1}A \ar[rd]^{}  &   & e_nA\ar[r] & S_n\ar[r]^{}  &  0 \\
  S_1 \ar[ru]^{} &     & b_{1}A \ar[ru]^{} &    & b_{n-2}A  \ar[ru]^{} &    & b_{n-1}A\ar[ru]^{} &  &  & }
\end{xy}
$$ 
$$
\begin{xy}
	\xymatrix@C-1pc@R-1pc{
	& e_nA \ar[rd]^{}   &         & ..... \ar[rd]^{} &     & e_{2}A \ar[rd]^{}  &   & e_1A\ar[r] & S_1\ar[r]^{}  &  0 \\
  S_n \ar[ru]^{} &     & a_{n-1}A \ar[ru]^{} &    & a_{2}A  \ar[ru]^{} &    & a_{1}A\ar[ru]^{} &  &  & }
\end{xy}
$$ 
Thus $Ext_A^{i}(S_n,S_n)=0$ for $i=1,...,2n-2$, but $Ext_A^{2n-1}(S_n,S_n)\neq 0$ follows from \ref{benson}.
That the algebra is in class $\mathcal{A}$ follows from the fact that $A$ is a symmetric algebra and $A \oplus S_n$ a generator-cogenerator.

\end{proof}
\begin{corollary}\label{classiA}
Let $A$ be a representation finite quiver algebra in the class $\mathcal{A}$ with $n \geq 2$ simple modules, then $A$ is isomorphic to the Auslander algebra of $K[x]/(x^3)$ or to $B(n,1,1,...,1)$.
\end{corollary}
\begin{proof}
Such an $A$ is clearly in $\mathcal{A}$.
Theorem \ref{donkinreiten} tells us that any other representation finite quiver algebra, which is quasi-hereditary with a duality has the form $B(n,\lambda_1,\lambda_2,...,\lambda_{n-2})$ or is isomorphic to the Auslander algebra of $K[x]/(x^3)$. But by the previous lemma, only $B(n,1,1,...,1)$ and the algebra isomorphic to the Auslander algebra of $K[x]/(x^3)$ have dominant dimension larger or equal two and thus those algebras are the only algebras which can be in the class $\mathcal{A}$.
\end{proof}

We note that the algebras $B(n,1,1,...,1)$ have a special importance, since every representation finite block of a Schur algebra is Morita equivalent to a $B(n,1,1,...,1)$ (see for example \cite{DoR}). For this reason we will study $B(n,1,1,...,1)$ now in more detail.

\begin{theorem}\label{schuralgex}
The algebra $B=B(n,1,1,...,1)$ is a higher Auslander algebra with $gldim(A)=domdim(A)=2n-2$ and $F(\Delta)=Dom_{n-1}=Proj_{n-1}$. The characteristic tilting module equals $T=eA \oplus S_1$, in case $eA$ is the minimal faithful projective-injective $A$-module.
\end{theorem}
\begin{proof}
We saw in the previous lemma that $B$ has dominant dimension $2n-2$. Now recall that a quasi-hereditary algebra with $n$ simple modules has global dimension at most $2n-2$, see for example \cite{DR} section 2. But the global dimension is always larger or equal to the dominant dimension for a non-selfinjective algebra. Thus $domdim(B)=2n-2$ forces $gldim(B)=domdim(B)$.
The equality $F(\Delta)=Dom_{n-1}=Proj_{n-1}$ follows from \hyperref[dualityhigh]{ \ref*{dualityhigh}}.
The statements about the characteristic tilting modules follows from $\Omega^{n-1}((1-e)B)=S_1$.
\end{proof}

\end{document}